\documentclass[a4paper,10pt]{amsproc}
\title{On the topology of convergence in measure, defined on the ring $\mathcal{M}(X,\mathscr{A},\mu)$}
\usepackage{amsfonts, amsmath, amssymb, graphicx, mathrsfs, geometry, enumitem, ulem, cancel}
\usepackage{newtxtext,newtxmath}
\usepackage[all]{xy}
\usepackage{xcolor}
\usepackage{hyperref}
\hypersetup{
	colorlinks=true,
	linkcolor={red!60!black},
	citecolor={blue!50!black},
	urlcolor={green!50!black}
}

\newdir{ >}{!/6pt/@{ }*:(1,0.2)@_{>}*:(1,-0.2)@^{>}}

\theoremstyle{plain}
\newtheorem{theorem}{Theorem}[section]
\newtheorem{lemma}[theorem]{Lemma}

\newtheorem{corollary}[theorem]{Corollary}
\theoremstyle{definition}
\newtheorem{definition}[theorem]{Definition}
\theoremstyle{definition}

\newtheorem{remark}[theorem]{Remark}
\newtheorem{counter example}[theorem]{Counter Example}

\newtheorem{example}[theorem]{Example}
\newtheorem{examples}[theorem]{Examples}
\newtheorem{question}[theorem]{Question}
\numberwithin{equation}{section}

\author[A. Dey]{Amrita Dey}	\address{Department of Pure Mathematics, University of Calcutta, 35, Ballygunge Circular Road, Kolkata 700019, West Bengal, India}	\email{deyamrita0123@gmail.com}



\begin{document}
	
	\title{On the topology of convergence in measure, defined on the ring $\mathcal{M}(X,\mathscr{A},\mu)$}

	%

	\thanks {The author is immensely grateful for the award of research fellowship provided by the University Grants Commission, New Delhi (NTA Ref. No. 221610014636).}

	\large
	\begin{abstract}
		For a probability measure space $(X,\mathscr{A},\mu)$, the topology $\mathcal{M}_\mu$, is defined on the ring $\mathcal{M}(X,\mathscr{A},\mu)$ of real-valued measurable functions on $(X,\mathscr{A},\mu)$ involving the notion of \textit{convergence in measure}. It turns out that if $f=g$ is assumed to be in the \textit{almost everywhere} sense, $\mathcal{M}_\mu$ is a completely metrizable space and is induced by the metric given by $\delta(f,g)=\mu(X\setminus Z(f-g))$, for $f,g\in \mathcal{M}(X,\mathscr{A},\mu)$. The notion of a measure being bounded away from zero is introduced and it is observed that a measure $\mu$ is bounded away from zero if and only if it is purely atomic and contains at most finitely many pairwise disjoint atoms. Topological properties, such as being a $P$-space, extremal disconnectedness and local connectedness of $\mathcal{M}_\mu$ are found to be equivalent to the underlying measure $\mu$ being bounded away from zero. The space $\mathcal{M}_\mu$ is proven to be never Lindel\"{o}f, and hence cannot be seperable, second countable or compact. It is established that $\mathcal{M}_\mu$ is connected (in fact, path-connected) if and only if $\mu$ is non-atomic and $\mathcal{M}_\mu$ is totally disconnected if and only if $\mu$ is purely atomic. The component of a point in $\mathcal{M}_\mu$ (which is found to be equivalent to the path-component and quasicomponent of that point in $\mathcal{M}_\mu$) is computed in a general setting.
	\end{abstract}	
	\maketitle
	
	\section{Introduction}
	
	Several topologies have already been defined on the ring $\mathcal{M}(X,\mathscr{A})$ of real-valued measurable functions, defined on the measurable space $(X,\mathscr{A})$, where $X$ is a non-empty set and $\mathscr{A}$ is a $\sigma$-algebra on $X$  \cite{SBag, RakeshDa, Dada, Estaji}. However, the topology introduced and studied in this manuscript differs significantly from the aforementioned topologies. In order to make this article self-contained, we recall few notions first.
	
	A measure $\mu$ on the measurable space $(X,\mathscr{A})$ is defined as a non-negative real-valued function on $\mathscr{A}$ which satisfies the following conditions:
	\begin{enumerate}[label=(\roman*)]
		\item $\mu(\emptyset)=0$ 
		\item For a sequence $\{A_n\colon n\in \mathbb{N} \}$ of pairwise disjoint sets  in $\mathscr{A}$,   $\displaystyle{\mu(\bigsqcup_{n=1}^{\infty}A_n)=\sum\limits_{n=1}^{\infty}\mu(A_n)}$.
	\end{enumerate}
	The triplet $(X,\mathscr{A},\mu)$ is called a measure space. Moreover, if $\mu(X)=1$, then $\mu$ is said to be a probability measure. Throughout this article, $\mu$ is always considered to be a probability measure. 
	Throughout this article, for $r\in \mathbb{R}$, $\boldsymbol{r}$ will denote the constant function on $X$ having value $r$ and for $A\subseteq X$, $\chi_A(x)=\begin{cases}
		1 &x\in A \\
		0 &x\in X\setminus A
	\end{cases}$. For each $f\in \mathcal{M}$, $Z(f)$ denotes the collection of all points in $X$ on which $f$ vanishes, that is, $Z(f)=\{x\in X\colon f(x)=0\}$. We say that $f,g\in \mathcal{M}$ are equal almost everywhere (``a.e.") with respect to $\mu$ on $X$ if $\mu(X\setminus Z(f-g))=0$. A sequence $\{f_n \}$ of measurable functions in $\mathcal{M}(X,\mathscr{A},\mu)$ is said to \textit{converge in measure} \cite{Folland} to an $f\in \mathcal{M}(X,\mathscr{A},\mu)$, i.e., ``$\{f_n\}\rightarrow f$ in measure'', if for each $c>0$, $$\lim\limits_{n\rightarrow \infty} \mu(\{x\in X\colon |f_n(x)-f(x)|>c \})=0.$$ Equivalently, $\{f_n\}$ is said to \text{converge in measure} to $f$ if for each $k\in \mathbb{N}$, $$\lim\limits_{n\rightarrow \infty} \mu(\{x\in X\colon |f_n(x)-f(x)|>\frac{1}{k} \})=0.$$ In this context, for each subset $A$ of $\mathcal{M}$, we define $\overline{A}$ as follows: \[f\in \overline{A} \text{ if } \text{ there exists a sequence } \{f_n\}\subseteq A \text{ such that }\{f_n\}\rightarrow f \text{ in measure}. \] One can verify using routine arguments that the family \[\{\mathcal{M} \setminus A\colon A\subseteq \mathcal{M} \text{ with }A=\overline{A} \} \] forms a topology on $\mathcal{M}$, which is termed the `\textit{topology of convergence in measure}'  on $\mathcal{M}(X,\mathscr{A},\mu)$, and shall henceforth be denoted by $\mathcal{M}_\mu$. Evidently, the concept of `convergence' in this topology is equivalent to that of the notion of `convergence in measure'.	Observe that, if $f=g$ a.e.\ on $X$ and a sequence $\{f_n\}$ in $\mathcal{M}(X,\mathscr{A},\mu)$ converges to $f$ in measure, then $\{f_n\}$ converges to $g$ in measure as well. To avoid such anomalies and for the purpose of this article, we assume that `$f=g$' means $f=g$ a.e. on $X$. Furthermore, the ring $\mathcal{M}(X,\mathscr{A},\mu)$ shall be denoted as $\mathcal{M}$ throughout this article. 
	
	Moreover, recall that the ring $\mathcal{M}$ is a Von-Neumann regular ring (A commutative ring with unity $R$ is said to be a Von-Neumann regular ring if for each $x\in R$, there exists $y\in R$ such that $x=x^2y$). A map $N\colon R\longrightarrow [0,1]$ on a Von-Neumann regular ring $R$ is said to be a pseudo-rank function \cite{Goodearl} if it satisfies the following conditions: \begin{enumerate}[label=(\roman*)]
		\item $N(1)=1$
		\item For $x,y\in R$, $N(xy)\leq N(x)$ and $N(xy)\leq N(y)$
		\item For $e,f\in R$ satisfying $e^2=1=f^2$ and $ef=0=fe$, $N(e+f)=N(e)+N(f)$. 
	\end{enumerate} Each pseudo-rank function induces a pseudometric $\delta$ on $R$ as $\delta(x,y)=N(x-y)$ for $x,y\in R$. $N$ is uniformly continuous on the pseudometric space $(R,\delta)$ \cite{Goodearl}. If additionally, $N(x)>0$ for all non-zero $x$ in $R$, then $N$ is said to be a rank function. Consequently, the pseudometric $\delta$ induced by $N$ forms a metric on $R$. One such rank function can be defined on $\mathcal{M}$ as follows: \[N_\mu(f):=\mu(X\setminus Z(f)). \] That $N_\mu$ is a rank function can be verified through elementary arguments. It is observed in this article that the metric $\delta$, induced by $N_\mu$ generates  the space $\mathcal{M}_\mu$. In other words, the topology of convergence in measure, $\mathcal{M}_\mu$ is a metrizable space. Moreover, it follows from \cite[Lemma 19.1]{Goodearl} that $\mathcal{M}_\mu$ is a topological ring. We dedicate this article to the study of this metric space in view of the underlying measure $\mu$. Prior to the discussion of this space, certain preliminaries are established.
	

	Section 2 of this article is devoted to building necessary mathematical tools for the development of this article. We recall several measure theoretic terms and results. It is a known fact that for a non-atomic measure $\mu$ and $r\in [0,1]$, there exists a measurable set $A_r$ with $\mu(A_r)=r$. It is verified that the sets of the form $A_r$ can be constructed in such a manner that whenever $r<s$, $A_r\subsetneqq A_s$, for $r,s\in [0,1]$. Also, it is realised that if $\mu$ is a purely atomic measure, then the range set of $\mu$ is at most a countable set with its complement dense in $[0,1]$. The result established by Johnson in \cite{J} which informs us that a measure $\mu$ can be decomposed into the sum of a purely atomic measure and a non-atomic measure, referred to as `the purely atomic part' and `the non-atomic part' of $\mu$ respectively. We define the concept of a measure being bounded away from zero and describe some connections between this notion and the atomicity of measure. The rest of this section is devoted to certain preliminary discussions on the topological space $\mathcal{M}_\mu$. At this point, the previously made assertion that the space $\mathcal{M}_\mu$ is induced by the metric $\delta$, where $\delta(f,g)=\mu(X\setminus Z(f-g))$, for any $f,g\in \mathcal{M}$; is established. In addition, it is observed that $\delta$  is a complete metric space and hence is also a \v{C}ech-complete space and a Baire space. Following this, an equivalence is obtained between phenomena of the measure $\mu$ being bounded away from zero and that of $\mathcal{M}_\mu$ being the discrete space.Combining this observation and the fact that $\mathcal{M}_\mu$ is metrizable, it is noted that $\mathcal{M}_\mu$  is a $P$-space and/or an extremally disconnected space if and only if $\mu$ is bounded away from zero. Moreover, unlike the well-known $m_\mu$-topology on $\mathcal{M}$ \cite{SBag}, the collection of units $U_\mu$ in $\mathcal{M}$ fails to be open in $\mathcal{M}_\mu$, for any $\mu$ which is not bounded away from zero. In other words, $U_\mu$ is not open in $\mathcal{M}_\mu$, unless it is the discrete space. These observations terminate this section.

	Section 3 deals with the discussions on compactness and related properties of the space $\mathcal{M}_\mu$. We first realise that $\mathcal{M}_\mu$ cannot be a Lindel\"{o}f space and since $\mathcal{M}_\mu$ is a metric space, it then follows that $\mathcal{M}_\mu$ cannot be a separable space or a second countable space either. Consequently, $\mathcal{M}_\mu$ is not a compact set. Moreover, we establish that if $\mu$ is not bounded away from zero (in particular, if $\mu$ is non-atomic), then any Lindel\"{o}f (resp.\ compact) set in $\mathcal{M}_\mu$ has empty interior. From this, we conclude that $\mathcal{M}_\mu$ is locally compact if and only if $\mu$ is bounded away from zero. Furthermore, it is proved the phenomenon that each compact set in $\mathcal{M}_\mu$ has finite cardinality is equivalent to $\mu$ being bounded away from zero. Additionally, if $\mu$ is not purely atomic, an uncountable Lindel\"{o}f set in $\mathcal{M}$ is constructed. Whether such a Lindel\"{o}f set exists in $\mathcal{M}_\mu$ for a purely atomic measure $\mu$ remains an unanswered question.
	
	In Section 4, we aim to discuss the concept of connectedness in the space $\mathcal{M}_\mu$. The section opens with the observation that the space $\mathcal{M}_\mu$ is a path-connected space whenever $\mu$ is a non-atomic measure. Moreover, it is shown that $\mathcal{M}_\mu$ becomes totally disconnected when the measure $\mu$ is chosen to be purely atomic. Having discussed the situations where $\mu$ is either purely atomic or non-atomic, we discuss an example where the measure $\mu$ is neither purely atomic nor non-atomic. Since $C_p(X)$ is a topological ring, our only concern is the component of $\boldsymbol{0}$. It is observed in the aforementioned example that the component of $\boldsymbol{0}$, denoted by $K_{\boldsymbol{0}}$ hereafter, is solely dependent on the purely atomic part of the measure $\mu$. Motivated by this example, the component of $\boldsymbol{0}$ has been computed for an arbitrary measure space as: $K_{\boldsymbol{0}}=\{f\in \mathcal{M}\colon \mu_1(X\setminus Z(f))=\{0\} \}$, where $\mu_1$ is the purely atomic part of the underlying measure $\mu$. It is further verified that $K_{\boldsymbol{0}}$ is also the path-component and the quasicomponent of $\boldsymbol{0}$ in $\mathcal{M}$. Consequently, the component, path-component and the quasicomponent of each point in the space $\mathcal{M}_\mu$ coincide. One can recall that this phenomenon occurs when a space is compact and Hausdorff or is locally connected. However, $\mathcal{M}_\mu$ is never a compact space. Moreover, for a measure which is purely atomic but not bounded away from zero, $\mathcal{M}_\mu$ is not locally connected. In fact, it is established that $\mathcal{M}_\mu$ is locally connected if and only if the purely atomic part of the underlying measure is either zero or bounded away from zero.

	\section{Prerequisites}
	
	We begin this section with the discussion of some measure theoretic concepts. A measurable set $A\in \mathscr{A}$ is said to be an atom \cite{J} if $\mu(A)>0$ and whenever $B\in \mathscr{A}$, either $\mu(A\cap B)=0$ or $\mu(A\setminus B)=0$. If each measurable set in $\mathscr{A}$ with positive measure contains an atom, then  the measure space $(X,\mathscr{A},\mu)$ is said to be purely atomic. If the measure space $(X,\mathscr{A},\mu)$ contains no atoms, then it is called non-atomic. We state a few examples. 
	\begin{examples} \label{eg}
		\		\begin{enumerate}[label=\arabic*.]
			\item \label{eg2}Consider $\mathscr{L}$ to be the 	$\sigma$-algebra of all Lebesgue measurable subsets of $[0,1]$ and $\mu_l$, the Lebesgue measure on $[0,1]$. Then the measure space $([0,1],\mathscr{L},\mu_l)$ is non-atomic.
			
			\item \label{eg3} Let $X$ be a non-empty set and $\mathscr{A}$, a $\sigma$-algebra on $X$. Let $p\in X$ be fixed. The Dirac measure $\delta_p$, at the point $p$, defined on $\mathscr{A}$ as: $\delta_p(A)=\begin{cases}
				0 &if\;p\in A \\
				1 &if\;p\in X\setminus A
			\end{cases}$ is a purely atomic measure on  $(X,\mathscr{A})$.
			
			\item \label{eg0} Let $X$ be an infinite set. Then there exists a countably infinite subset $N=\{x_n\colon n\in \mathbb{N} \}$ of $X$. Suppose $\mathscr{A}$ is a $\sigma$-algebra on $X$ such that $\{x_n \}\in \mathscr{A}$ for each $n\in \mathbb{N}$. On the measurable space $(X,\mathscr{A})$, define the measure $\mu_N$ as $\mu_N(A)=0$ if $A\cap N=\emptyset$ and whenever $A\cap N\neq \emptyset$, $\mu_N(A)=\sum\limits_{n\in S}\frac{1}{2^n}$, where $S=\{n\in \mathbb{N}\colon x_n\in A\cap N \}$. Then for each $n\in \mathbb{N}$, $\{x_n \}$ is an atom and thus, this measure space is a purely atomic.
		\end{enumerate}
	\end{examples}
	
	The notations that we have used in the above examples shall be prevalent throughout this article. Sierpi\'{n}ski established the following result for a non-atomic measure space.
	
	\begin{theorem} \cite{S} \label{tnonatomic}
		Let $\mu$ be a non-atomic measure on the measurable space $(X,\mathscr{A})$ and $A\in \mathscr{A}$ be such that $\mu(A)$ is a positive real number. Then for each $r\in [0,\mu(A)]$, there exists $A_r\in \mathscr{A}$ such that $\mu(A_r)=r$.
	\end{theorem}
	
	We extend the above result a little to obtain a kind of ordering in the sets of the form $A_r$, which shall be very useful in the subsequent sections.
	
	\begin{theorem} \label{nonatomic}
		Let $\mu$ be a non-atomic measure on a measurable space $(X,\mathscr{A})$. Then, for each $r\in [0,1]$, we can associate an $A_r\in \mathscr{A}$ such that $\mu(A_r)=r$ and whenever $r\leq s$, $A_r\subseteq A_s$.
		
	\end{theorem}
	
	\begin{proof}
		Consider the collection $\mathscr{F}$ of all functions $A\colon D\longrightarrow \mathscr{A}$ where $D\subseteq [0,1]$, $\mu(A(r))=r$ for each $r\in D$ and whenever $r,s\in D$ with $r\leq s$, $A(r)\subseteq A(s)$. The existence of such a function can be shown by considering $D=\{0,1\}$ with $A(0)=\emptyset$ and $A(1)=X$. The non-empty set $\mathscr{F}$ forms a partially ordered set with the relation that $A_1\leq A_2$ if $D_1\subseteq D_2$ and for all $r\in D_1$, $A_1(r)=A_2(r)$. 
		
		Consider a chain $\{A_\alpha\colon \alpha\in \Lambda \}$ in $\mathscr{F}$, where each $A_\alpha$ has domain $D_\alpha$. Define $A\colon \bigcup\limits_{\alpha\in \Lambda}D_\alpha \longrightarrow \mathscr{A}$ as $A(r)=A_\alpha(r)$ whenever $r\in D_\alpha$. Then it is evident that $A$ is an upper bound of the chain $\{A_\alpha\colon \alpha\in \Lambda \}$. So, by Zorn's Lemma $\mathscr{F}$ has a maximal element. 
		
		We assert that the domain of a maximal element is $[0,1]$. To see this, let $A\in \mathscr{F}$ be an element with domain $D\subsetneqq [0,1]$, then $[0,1]\setminus D\neq \emptyset$. If $0$ or $1$ is not in $D$, then we can extend the domain $D$ of $A$ to $D\cup \{0\}$ or $D\cup \{1\}$ and map $0$ to $\emptyset$ or $1$ to $X$ respectively.	Now, consider $0,1\in D$ and $c\in [0,1]\setminus D$. Define $D_{<c}=\{r\in D\colon r<c \}$ and $D_{>c}=\{r\in D\colon r>c \}$. If $r=\sup D_{<c}$, then there exists an increasing sequence $\{r_n\in D_{<c}\}$ converging to $r$. Define $A_r=\bigcup\limits_{n\in \mathbb{N}}A(r_n) $. Similarly, if $r=\inf D_{>c}$, then there exists a decreasing sequence $\{r_n\in D_{>c}\}$ converging to $r$ and we define $A_r=\bigcap\limits_{n\in \mathbb{N}}A(r_n) $. The map $A'\colon D\cup \{r\}\longrightarrow \mathscr{A}$ defined as   $A'(s)=A(s)$ for all $s\in D$ and $A'(r)=A_r$ is a member of $\mathscr{F}$ with $A'$ strictly greater than $A$. Finally, we assume that $0,1,\sup D_{<c}, \inf D_{>c}\in D$ for all $c\in [0,1]\setminus D$. Now, fix $c\in [0,1]\setminus D$. Let $a_0=\sup D_{<c}$ and $a_1=\inf D_{>c}$. Then $A(a_0)\subseteq A(a_1)$, $a_0<c<a_1$ and $\mu(A(a_1)\setminus A(a_0))=a_1-a_0$. Since $\mu(A(a_1)\setminus A(a_0))$, it follows from Theorem \ref{tnonatomic} that there exists a $B\in \mathscr{A}$ such that $\mu(B)=c-a_0$ and $B\subseteq A(a_1)\setminus A(a_0)$. Define $A_c=B\sqcup A(a_0)$. Then $\mu(A_c)=c$ and the map $A'\colon D\cup \{c\}\longrightarrow \mathscr{A}$ defined as   $A'(r)=A(r)$ for all $r\in D$ and $A'(c)=A_c$ is a member of $\mathscr{F}$ with $A'$ strictly greater than $A$. This ensures that any member $A\in \mathscr{F}$ having a domain which is properly contained in $[0,1]$ cannot be a maximal element.
	\end{proof}

	We note that there exist measures which are neither purely atomic nor non-atomic as can be observed in the next example.
	
	\begin{example} \label{eg4}
		Consider the measurable space $([0,1],\mathscr{L})$ and the measures $\mu_l$ and $\delta_0$ on $([0,1],\mathscr{L})$.
		Then $\mu=\frac{1}{2}(\mu_l+\delta_0)$ is a measure on $([0,1],\mathscr{L},\mu)$. Clearly, $\{0 \}$ is an atom in the measure space $([0,1],\mathscr{L},\mu)$, but the positive measurable set $[\frac{1}{2},1]$ contains no atoms. Consequently, $\mu$ is neither purely atomic nor non-atomic.
	\end{example} 
	
	We recall that if $\mu_1$ and $\mu_2$ are two measures on $(X,\mathscr{A})$, then $\mu_1$ is said to be `$\mathcal{S}$-singular' with respect to $\mu_2$, denoted by $\mu_1\mathcal{S}\mu_2$, if given any $E\in \mathscr{A}$, there exists $F\in \mathscr{A}$ with $F\subseteq E$ such that $\mu_1(E)=\mu_1(F)$ and $\mu_2(F)=0$ \cite{J}. Due to Johnson, we have the  following results.
	
	\begin{theorem}\cite[Theorem 2.1]{J} \label{tdecompose}
		Let $\mu$ be a measure on the measurable space $(X,\mathscr{A})$. Then $\mu$ can be expressed as $\mu=\mu_1+\mu_2$ with $\mu_1 \mathcal{S} \mu_2$ and $\mu_2 \mathcal{S} \mu_1$, where $\mu_1$ is purely atomic and $\mu_2$ is non-atomic.
	\end{theorem}
	
	In light of the above result, we shall refer to $\mu_1$ as the `\textit{purely atomic part}' of $\mu$ and  $\mu_2$ as the `\textit{non-atomic part}' of $\mu$.
	
	
	\begin{theorem} \cite[Theorem 2.2]{J} \label{Johnson}
		If $(X,\mathscr{A},\mu)$ is a purely atomic measure space and $\mu(E)>0$, then there exists a countable collection of pairwise disjoint atoms $\{E_k\}_{k\in \mathbb{N}}$, each contained in $E$, such that $\mu(E)=\mu(\bigsqcup\limits_{n\in \mathbb{N}}E_k)$.
		
	\end{theorem}
	
	We observe that if $\mu$ is a measure which is not purely atomic, then its range contains an interval.
	\begin{theorem} \label{atomic}
		Let $\mu$ be a probability measure on a measurable space $(X,\mathscr{A})$. Then the following statements are equivalent.
		
		\begin{enumerate}[label=\arabic*.]
			\item $\mu$  is purely atomic.
			\item $\mu(\mathscr{A})$ is at most countable.
			\item $[0,1]\setminus \mu(\mathscr{A})$ is dense in $[0,1]$.
		\end{enumerate}  
		
	\end{theorem}
	
	\begin{proof}
		By Theorem \ref{tdecompose}, $\mu=\mu_1+\mu_2$ with $\mu_1 \mathcal{S} \mu_2$ and $\mu_2 \mathcal{S} \mu_1$, where $\mu_1$ is purely atomic and $\mu_2$ is non-atomic.
		
		First assume that $\mu$ is purely atomic. By Theorem \ref{Johnson}, there exists a countable collection of pairwise disjoint atoms $\{E_k\}_{k\in \mathbb{N}}$ in $X$ such that $\mu(X)=\mu(\bigsqcup\limits_{n\in \mathbb{N}}E_k)=\sum\limits_{k\in \mathbb{N}}\mu( E_k)$. We assert that for each atom $A$ in $X$, there exists a unique $n\in \mathbb{N}$ such that $\mu(A)=\mu(E_n)$. Indeed, $\mu(A)=\mu(A\cap \bigsqcup\limits_{k\in \mathbb{N}}E_k)=\sum\limits_{k\in \mathbb{N}}\mu(A\cap E_k)$. Since $\mu(A)>0$, there exists $n\in \mathbb{N}$ such that $\mu(A\cap E_n)>0$. That this $n$ is unique follows from the fact that $A$ is an atom and $\{E_k\}_{k\in \mathbb{N}}$ is a collection of pairwise disjoint atoms. Therefore, $\mu(A)=\mu(A\cap E_n)=\mu(E_n)$. Now, consider a measurable set $E\in \mathscr{A}$ with $\mu(E)>0$. Again by Theorem \ref{Johnson}, there exists  a countable collection of pairwise disjoint atoms $\{F_k\}_{k\in \mathbb{N}}$ in $X$ with $\mu(E)=\mu(\bigsqcup\limits_{n\in \mathbb{N}}F_k)=\sum\limits_{k\in \mathbb{N}}\mu(F_k)$. Now, for each $k\in \mathbb{N}$, there exists $n_k\in \mathbb{N}$ with $\mu(F_k)=\mu(E_{n_k})$ and so $\mu(E)=\sum\limits_{k\in \mathbb{N}}\mu(E_{n_k})$. Thus, measure of a measurable set in $X$ lies in the set $\{\sum\limits_{n\in A}\mu(E_n)\colon  A\subseteq \mathbb{N}  \}$, which is at most a countable set.
		
		Now consider $\mu$ to be not purely atomic, then $\mu_2$ is non-zero and so $\mu_2(X)>0$. For each $r\in [0,\mu_2(X)]$, there exists $A_r\in \mathscr{A}$ with $\mu_2(A_r)=r$ (by Theorem \ref{tnonatomic}). Since $\mu_2\mathcal{S}\mu_1$, for each $A_r$, there exists $F_r\in \mathscr{A}$ with $F_r\subseteq A_r$ such that $\mu_2(A_r)=\mu_2(F_r)$ and $\mu_1(F_r)=0$. Therefore, $\mu(F_r)=r$ for each $r\in [0,\mu_2(X)]$. This ensures that $\mu(\mathscr{A})$ contains $[0,\mu_2(X)]$.
	\end{proof}
	
	For the purpose of this article, we define the following crucial class of measures.
	
	\begin{definition}
		A measure $\mu$ is defined to be bounded away from zero if there exists $\lambda>0$ such that for all $A\in \mathscr{A}$, either $\mu(A)=0$ or $\mu(A)\geq \lambda$.
	\end{definition} 
	We note some connections between the concept of a measure being bounded away from zero and that of the atomicity of a measure.
	
	\begin{theorem}\label{Th1.1}
		The following assertions hold for a measure space $(X,\mathscr{A},\mu)$:
		\begin{enumerate}[label=\arabic*.]
			\item If $\mu$ is a non-atomic measure, then it cannot be bounded away from zero.
			\item If $\mu$ is bounded  away from zero, then it is a purely atomic measure. \label{Th1.1.2}
		\end{enumerate}
	\end{theorem} 
	\begin{proof}
		\	\begin{enumerate}[label=\arabic*.]
			\item This follows from Theorem \ref{tnonatomic}.
			\item By Theorem \ref{tdecompose}, $\mu$ can be decomposed as $\mu=\mu_1+\mu_2$ with $\mu_1 \mathcal{S} \mu_2$ and $\mu_2 \mathcal{S} \mu_1$, where $\mu_1$ is purely atomic and $\mu_2$ is non-atomic. Assume that $\mu$ is not purely atomic. Then, $\mu_2$ is non-zero. Proceeding as in the proof of Theorem \ref{atomic}, $\mu(\mathscr{A})\supseteq [0,\mu_2(X)]$. Therefore, $\mu$ takes values arbitrarily close to zero and hence is not bounded away from zero.
		\end{enumerate}
	\end{proof}
	
	We note that not all purely atomic measures are bounded away from zero. Indeed, Example \ref{eg}(\ref{eg0}) defines a purely atomic measure which is not bounded away from zero. In fact, we observe something stronger.
	
	\begin{theorem} \label{tbdd}
		Let $(X,\mathscr{A},\mu)$ be a measure space. Then $\mu$ is bounded away from zero if and only if $\mu$ is purely atomic and $(X,\mathscr{A},\mu)$ contains at most finitely many pairwise disjoint atoms.
	\end{theorem}
	\begin{proof}
		Assume that $\mu$ is bounded away from zero. That it is purely atomic follows from Theorem \ref{Th1.1}(\ref{Th1.1.2}). Now, let $\lambda>0$ be such that for all positive measurable sets $A\in \mathscr{A}$, $\mu(A)\geq \lambda$. If possible let there are infinitely many pairwise disjoint atoms in the measure space. By Theorem \ref{Johnson}, there exists a countably infinite collection of pairwise disjoint atoms $\{E_n\colon n\in \mathbb{N} \}$ such that $\mu(X)=\mu(\bigsqcup\limits_{n\in \mathbb{N}}E_n)=\sum\limits_{n\in \mathbb{N}}\mu(E_n)$. But $\mu(E_n)\geq \lambda$ for each $n\in \mathbb{N}$ and so the series $\sum\limits_{n\in \mathbb{N}}\mu(E_n)$ diverges to infinity, which contradicts that $\mu(X)=1$.
		
		Conversely, let $\{E_i\colon i=1,2,\cdots, n \}$,  $n\in \mathbb{N}$ be a collection of pairwise disjoint atoms such that $\mu(X\setminus \bigsqcup\limits_{i=1}^nE_i)=0$. Let $\lambda=\min \{\mu(E_i)\colon i=1,2,\cdots, n \}$. Then $\lambda>0$. Now, proceeding as in the proof of Theorem \ref{atomic}, for each $E\in \mathscr{A}$ with $\mu(E)>0$, $\mu(E)=\sum\limits_{i\in A}\mu(E_i)$ where $A$ is a non-empty subset of $\{ 1,2,\cdots,n\}$ and so $\mu(E)\geq \lambda$.  Thus, $\mu$ is bounded away from zero.
	\end{proof}
	
	Having established the necessary measure theoretic framework, we discuss a few preliminary properties of the space $\mathcal{M}_\mu$. First of all, it has been mentioned that this space can be induced through the rank function $N_\mu$. We shall verify this assertion. In order to do that, we revisit the notion of Cauchy in measure.
	
	\begin{definition} \cite{Folland}
		A sequence $\{f_n\}$ in $\mathcal{M}(X,\mathscr{A},\mu)$ is said to be \textbf{Cauchy in measure} provided for each natural number $k$ and each $\epsilon>0$, there exists $N\in \mathbb{N}$ such that for all $n,m\geq N$, \[\mu(\{x\in X\colon |f_n(x)-f_m(x)|> \frac{1}{k} \})<\epsilon. \] 
	\end{definition}
	
	Pertaining to the above definition, we reserve the notation $[f,g,k]$ to denote $\{x\in X\colon |f(x)-g(x)|>\frac{1}{k} \}$, for $f,g\in \mathcal{M}(X,\mathscr{A},\mu)$ and $k\in \mathbb{N}$. It follows that $X\setminus Z(f-g)=\bigcup\limits_{k\in \mathbb{N}}[f,g,k]$, for any  $f,g\in \mathcal{M}(X,\mathscr{A},\mu)$. Moreover, $\{[f,g,k] \}_{k=1}^\infty$ is an increasing sequence of measurable sets and hence, \[\delta(f,g)=\mu(X\setminus Z(f-g))=\lim\limits_{k\rightarrow \infty} \mu([f,g,k]). \] These observations lead us to the equivalence between the concepts of convergence (or, Cauchy) in measure and in the metric space induced by the rank function $N_\mu$. To achieve this efficiently, we first recall the following result.
	
	\begin{lemma} \cite[Theorem 2.30]{Folland} \label{lemma_complete}
		If a sequence $\{f_n\}$, in $\mathcal{M}$, is Cauchy in measure, there exists $f\in \mathcal{M}$ such that $\{f_n\}\rightarrow f$ in measure.
	\end{lemma}
	
	For the sake of convenience, we shall denote the metric space induced by the rank function $N_\mu$ by $(\mathcal{M},N_\mu)$.
	
	\begin{theorem} \label{tconvergence_equivalence}
		Suppose $\{f_n\}$ is a sequence in $\mathcal{M}$ and $f\in \mathcal{M}$. Then
		\begin{enumerate}[label=\arabic*.]
			\item $\{f_n\}$ is Cauchy in $(\mathcal{M},N_\mu)$ if and only if it is Cauchy in measure.
			\item $\{f_n\}$ converges to $f$ in $(\mathcal{M},N_\mu)$ if and only if $\{f_n\}\rightarrow f$ in measure.
		\end{enumerate}
	\end{theorem}
	
	\begin{proof}
		Note that $[f,g,k]\subseteq X\setminus Z(f-g)$, for any $f,g\in \mathcal{M}$ and $k\in \mathbb{N}$. Therefore, it is evident that if $\{f_n\}$ is Cauchy in $(\mathcal{M},N_\mu)$ (resp. converges to $f$ in $(\mathcal{M},N_\mu)$), then it is Cauchy in measure (resp. $\{f_n\}\rightarrow f$ in measure). 
		
		Now, suppose $\{f_n\}\rightarrow f$ in measure. Then, for each $k\in \mathbb{N}$, $\lim\limits_{n\rightarrow \infty}\mu([f_n,f,k])=0$. Recall that for each $n\in \mathbb{N}$, \[0\leq \mu([f_n,f,1])\leq \mu([f_n,f,2])\leq \cdots \rightarrow \mu(X\setminus Z(f_n-f)). \] Consequently, $\lim\limits_{n\rightarrow \infty}\mu(X\setminus Z(f_n-f))=0$, i.e., $\delta(f_n,f)\rightarrow 0$ and so $\{f_n\}$ is convergent in $(\mathcal{M},N_\mu)$.
		
		Finally, suppose $\{f_n\}$ is Cauchy in measure. Then by Lemma \ref{lemma_complete}, there exists $f\in \mathcal{M}$ such that $\{f_n\}\rightarrow f$ in measure, and so $\{f_n\}$ converges to $f$ in the space $(\mathcal{M},N_\mu)$. As $(\mathcal{M},N_\mu)$ is a metric space, it follows that $\{f_n\}$ is Cauchy in $(\mathcal{M},N_\mu)$.
	\end{proof}
	
	That the notions of `Cauchy in measure' (resp.\ `convergence in measure') and `Cauchy in $\mathcal{M}_\mu$' (resp.\ `convergence in $\mathcal{M}_\mu$') are equivalent is evident. The aforementioned equivalence between $\mathcal{M}_\mu$ and $(\mathcal{M},N_\mu)$ shall now be established. Note that, as usual, $B(f,\epsilon)=\{g\in \mathcal{M}\colon \delta(f,g)<\epsilon \}$, for any $f\in \mathcal{M}$ and $\epsilon>0$. Also, whenever $\epsilon>1$, $B(f,\epsilon)=\mathcal{M}$. So, without loss of generality, we shall always assume $\epsilon\leq 1$.
	
	\begin{theorem} \label{tConvergenceTopology=rankfunctionMetric}
		The topology of convergence in measure $\mathcal{M}_\mu$ is induced by the rank function $N_\mu$. 
	\end{theorem}
	
	\begin{proof}
		Suppose $X\setminus A$ is an open set in the topology of pointwise convergence, where $A=\overline{A}$ and let $f\in X\setminus A$. We assert that there exists $n\in \mathbb{N}$ such that $B(f,\frac{1}{n})\subseteq X\setminus A$. If not, then for each $n\in \mathbb{N}$, there is an $f_n\in B(f,\frac{1}{n})\cap A$. By Theorem \ref{tconvergence_equivalence}, $\{f_n\}\rightarrow f$ in measure. Since $\{f_n\}\subseteq A$, $f\in \overline{A}=A$, which is a contradiction. 
		
		Now, let $f\in \mathcal{M}$ and $\epsilon\in (0,1]$. To see that $B:=B(f,\epsilon)$ is open in $\mathcal{M}_\mu$, we need to show that $X\setminus B=\overline{X\setminus B}$. Let $g\in \overline{X\setminus B}$. Then, there exists a sequence $\{g_n\}\subseteq X\setminus B$ such that $\{g_n\}\rightarrow g$ in measure. By Theorem \ref{tconvergence_equivalence}, $\{g_n\}$ converges to $g$ in $(\mathcal{M},N_\mu)$. Consequently, $g\in X\setminus B$. Therefore, $B$ is open in the space $\mathcal{M}_\mu$.
	\end{proof}
	
	The subsequent result is an immediate consequence of the above theorem.
	
	\begin{corollary} \label{tMetrizable}
		The topology of convergence in measure, $\mathcal{M}_\mu$ is induced by the metric $\delta$, defined as: \[\delta(f,g)=\mu(X\setminus Z(f-g)), \text{ for }f,g\in \mathcal{M}. \] Consequently, $\mathcal{M}_\mu$ a metrizable space. 
	\end{corollary}
	
	In light of Lemma \ref{lemma_complete}, one can further observe that the space $\mathcal{M}$ is completely metrizable.
	
	\begin{theorem} \label{tcompletely_metrizable}
		The ring $\mathcal{M}$, equipped with the topology of convergence in measure is completely metrizable. In particular, $\mathcal{M}_\mu \equiv (\mathcal{M},\delta)$ is a complete metric space.
	\end{theorem}
	
	The next corollary can be deduced by unifying \cite[Theorem 3.9.3,\ Theorem 4.3.26]{Engelking} and the above theorem. 
	
	\begin{corollary}
		The space $\mathcal{M}_\mu$ possesses the following properties:
		\begin{enumerate}[label=\arabic*.]
			\item $\mathcal{M}_\mu$ is a \v{C}ech-complete space.
			\item $\mathcal{M}_\mu$ is a Baire space.
		\end{enumerate}
	\end{corollary}
	
	At this point, note that if $\mu$ is bounded away from zero by a positive real number $\lambda$, then for any $\epsilon\in (0,\lambda)$ and any $f\in \mathcal{M}$, $B(f,\epsilon)=\{f \}$. Therefore, the metric space $\mathcal{M}_\mu$ reduces to the discrete space, whenever $\mu$ is bounded away from zero. Additionally, the converse also holds.
	
	\begin{theorem} \label{tDiscrete}
		The space $\mathcal{M}_\mu$ is discrete if and only if the underlying measure $\mu$ is bounded away from zero.
	\end{theorem}
	
	\begin{proof}
		That $\mathcal{M}_\mu$ is a discrete space when $\mu$ is bounded away from zero, has already been discussed. So, suppose $\mathcal{M}_\mu$ is the discrete space. Then $\{\boldsymbol{0}\}$ is open in $\mathcal{M}_\mu$ and so, there exists $\lambda>0$ such that $B(\boldsymbol{0},\lambda)=\{\boldsymbol{0} \}$. Now, let $A$ be a measurable set with $\mu(A)>0$. Then $\chi_A\notin B(\boldsymbol{0},\lambda)$, as $\mu(A)=\delta(\chi_A,\boldsymbol{0})$. This ensures that $\mu(A)\geq \lambda$. 
	\end{proof}
	
	Since $\mathcal{M}_\mu$ is a metric space, the following corollary is immediate.
	
	\begin{corollary} \label{tPSpace}
		The following statements are equivalent.
		
		\begin{enumerate}[label=\arabic*.]
			\item $\mu$ is bounded away from zero.
			\item $\mathcal{M}_\mu$ is a $P$-space.
			\item $\mathcal{M}_\mu$ is extremally disconnected.
		\end{enumerate}
	\end{corollary}
	
	Now, we recall that $\mathcal{M}$ equipped with the $u_\mu$-topology is not, in general, a topological ring (see \cite{RakeshDa}). This brings out a contrast between the well-known $u_\mu$-topology on $\mathcal{M}$ and the space $\mathcal{M}_\mu$, as the latter always forms a topological ring. We next observe a noteworthy difference between the $m_\mu$-topology and the space $\mathcal{M}_\mu$. Recall that a function $f\in \mathcal{M}$ is a unit if and only if $\mu(Z(f))=0$. Let $U_\mu$ denote the set of all units in $\mathcal{M}$. Furthermore, it is observed in \cite[Theorem 3.2]{SBag} that $U_\mu$ is open in the $m_\mu$-topology on $\mathcal{M}$. On the other hand, we observe that for a large class of measure spaces, the collection $U_\mu$ fails to be open the space $\mathcal{M}_\mu$.
	
	\begin{theorem} \label{t3.1}
		Let $\mu$ be a non-atomic measure on a measurable space $(X,\mathscr{A})$. Then $U_\mu$ is not an open set in $\mathcal{M}_\mu$.
	\end{theorem}
	\begin{proof}
		Let $f\in U_\mu$ and $\epsilon\in (0,1]$ be chosen arbitrarily. Then $\mu(Z(f))=0$ and so, $\mu(X\setminus Z(f))=1$. Since $\mu$ is non-atomic, there exists a measurable set $A\subseteq X\setminus Z(f)$ such that $\mu(A)=\frac{\epsilon}{2}$. Define $g\colon X\longrightarrow \mathbb{R}$ as $g(x)=\begin{cases}
			0, &x\in A\\
			f(x), &x\notin A
		\end{cases}.$ Therefore, $\mu(Z(g))=\mu(A)\neq 0$ which implies that $g\notin U_\mu$. Now, $X\setminus Z(f-g)= A$ and so $\mu(X\setminus Z(f-g))= \frac{\epsilon}{2}<\epsilon$. Thus, $B(f,\epsilon)\nsubseteq U_\mu$.
	\end{proof}
	
	In fact, the openness of $U_\mu$ in $\mathcal{M}_\mu$ characterises the measure $\mu$ as can be seen in the next result.
	
	\begin{theorem} \label{t3.2}
		Let $(X,\mathscr{A},\mu)$ be a measure space. Then $\mu$ is bounded away from zero if and only if $U_\mu$ is an open set in $\mathcal{M}_\mu$.
	\end{theorem}
	\begin{proof}
		Let $\mu$ be bounded away from zero. Then it follows from Theorem \ref{tDiscrete} that $U_\mu$ is open in $\mathcal{M}_\mu$. 
		
		Conversely, let $U_\mu$ be open. If possible let for each $\epsilon>0$, there exists $A_\epsilon\in \mathscr{A}$ such that $0<\mu(A_\epsilon)<\epsilon$. Consider the point $\boldsymbol{1}\in U_\mu$. Then $B(\boldsymbol{1},\epsilon)\subseteq U_\mu$ for some $\epsilon>0$. Define $g\colon X\longrightarrow \mathbb{R}$ as follows: \[g(x)=\begin{cases}
			0 &if\;x\in A_\epsilon \\
			1 &otherwise
		\end{cases}. \] Then $g\in B(\boldsymbol{1},\epsilon)$. But $\mu(Z(g))=\mu(A_\epsilon)>0$ which ensures that $g\notin U_\mu$. Thus, $B(\boldsymbol{1},\epsilon)\nsubseteq U_\mu$ which is a contradiction.
	\end{proof}
	
	Combining this observation with Theorem \ref{tDiscrete}, we have the following corollary.
	
	\begin{corollary}
		The following statements are equivalent for a probability measure space $(X,\mathscr{A},\mu)$.
		
		\begin{enumerate}[label=\arabic*.]
			\item $U_\mu$ is an open set in $\mathcal{M}_\mu$.
			\item $\mu$ is bounded away from zero.
			\item $\mathcal{M}_\mu$ is a discrete space.
		\end{enumerate}
	\end{corollary}



	\section{Compactness and Lindel\"{o}fness in $\mathcal{M}_\mu$}

	We begin with a discussion of Lindel\"{o}fness in the space $\mathcal{M}_\mu$.  We first observe that 
	$\delta(\boldsymbol{r},\boldsymbol{s})=1$ whenever $r\neq s$ and $r,s\in \mathbb{R}$. Thus, if we consider an open cover $\{B(f,\frac{1}{4})\colon f\in \mathcal{M}_\mu \}$ of $\mathcal{M}_\mu$, it cannot have a countable subcover, as $\mathbb{R}$ is an uncountable set. Hence, $\mathcal{M}_\mu$ is never a Lindel\"{o}f space. In this context, recall that the notions of separability,  second countability and Lindel\"{o}fness are equivalent topological properties for a metrizable space. We unify these discussions in the following theorem.
	
	\begin{theorem} \label{tL}
		The following assertions hold for any measure space $(X,\mathscr{A},\mu)$.
		\begin{enumerate}[label=\arabic*.]
			\item \label{L1} $\mathcal{M}_\mu$ is not a Lindel\"{o}f space.
			\item \label{L2} $\mathcal{M}_\mu$ is not a separable space.
			\item \label{L3} $\mathcal{M}_\mu$ is not a second countable space.
		\end{enumerate}
	\end{theorem}

	The following corollary is immediate.
	
	\begin{corollary} \label{tcompact}
		$\mathcal{M}_\mu$ is not a compact space.
	\end{corollary}
	
	While we are on this subject, we aim to identify compact (and Lindel\"{o}f) subsets of $\mathcal{M}_\mu$. We first recall that if $\mu$ is bounded away from zero, then the space $\mathcal{M}_\mu$ is nothing but the discrete space. So, throughout this section, we shall assume that  $\mu$ is not bounded away from zero. 
	
	\begin{theorem} \label{tL2}
		Let $L$ be a Lindel\"{o}f subset of $\mathcal{M}_\mu$. Then $int\;L=\emptyset$. 
	\end{theorem}
	
	\begin{proof}
		If possible let there exist $f\in int\; L$. Then $B(f,\epsilon)\subseteq L$ for some $\epsilon\in (0,1]$. Since $\mu$ is not bounded away from zero, there exists $A\in \mathscr{A}$ such that $0<\mu(A)<\epsilon$. For each $r\in \mathbb{R}\setminus \{0\}$, define $f_r\colon X\longrightarrow \mathbb{R}$ as $f_r(x)=\begin{cases}
			f(x)-r &if\;x\in A \\
			f(x) &if\;x\notin A
		\end{cases}$. Then, $\delta(f,f_r)=\mu(A)<\epsilon$ and so $f_r\in B(f,\epsilon)\subseteq L$ for all $r\neq 0$. Also, $\delta(f_r,f_s)=\mu(A)$, whenever $r\neq s$. Hence, the open cover $\{B(g,\frac{\mu(A)}{4})\colon g\in L \}$ of $L$ has no countable subcover, as each $f_r$ lies in exactly one member of the cover.
	\end{proof}
	
	The following observations are immediate.
	\begin{corollary} \label{ccompact}
		Let  $K$ be a compact subset of $\mathcal{M}_\mu$. Then $int\;K=\emptyset$. 
	\end{corollary}

	\begin{corollary} \label{clocal}
		$\mathcal{M}_\mu$ is locally compact if and only if $\mu$ is bounded away from zero.
	\end{corollary}
	
	\begin{proof}
		If $\mu$ is bounded away from zero, then it follows from Theorem \ref{tDiscrete}  that every subset of $\mathcal{M}$ is open and so, $\mathcal{M}_\mu$ is locally compact.  The converse follows from Corollary \ref{ccompact}.
	\end{proof}

	If $\mu$ is a non-atomic measure, then it cannot be bounded away from zero. The following corollary is thus immediate from Theorem \ref{tL2}.
	\begin{corollary} \label{cL}
		Let $\mu$ be a non-atomic measure and $L$, a Lindel\"{o}f (resp.\ compact) subset of $\mathcal{M}_\mu$. Then $int\;L=\emptyset$. 
	\end{corollary}

	It is obvious that if $\mu$ is bounded away from zero, then the only Lindel\"{o}f (resp. compact) subsets of $\mathcal{M}_\mu$ are the countable (resp. finite) sets. The converse also holds for compact subsets of $\mathcal{M}_\mu$ as has been established in the following theorem.
	
	\begin{theorem} \label{tcompact2}
		Let $\mu$ be not bounded away from zero. Then $\mathcal{M}_\mu$ contains an infinite compact set.
	\end{theorem} 
	
	\begin{proof}
		
		For each $n\in \mathbb{N}$, we associate $k_n\in \mathbb{N}$ and $A_n\in \mathscr{A}$ inductively as follows: $A_1\in \mathscr{A}$ is such that $0<\mu(A_1)<1$ and $k_1=1$. Then there exists $k_2\geq 2$ such that $\frac{1}{k_2}<\mu(A_1)$, $A_2\in \mathscr{A}$ is chosen such that $0<\mu(A_2)<\frac{1}{k_2}$. Continuing this process inductively, we have an increasing sequence $\{k_n\}\subseteq \mathbb{N}$ and a sequence of measurable sets $\{A_n\}$ such that $k_n\geq n$ and $\frac{1}{k_{n+1}}<\mu(A_n)<\frac{1}{k_n}$ for each $n\in \mathbb{N}$. Clearly, $\lim\limits_{n \rightarrow \infty} k_n=0$ and so $\{\frac{1}{k_n}\colon n\in \mathbb{N} \}\cup \{0\}$ is a compact set.

		With each $n\in \mathbb{N}$, we associate a function $f_n\in \mathcal{M}_\mu$, defined as $f_n(x)=\begin{cases}
			1 &if\;x\in A_n\\
			0 &otherwise
		\end{cases}$. Let $K=\{f_n\colon n\in \mathbb{N} \}\cup \{\boldsymbol{0} \}$. We now assert that the function $\phi\colon \{\frac{1}{k_n}\colon n\in \mathbb{N} \}\cup \{0\} \longrightarrow K$, defined as $\phi(\frac{1}{k_n})=f_n$ for each $n\in \mathbb{N}$ and $\phi(0)=\boldsymbol{0}$ is a continuous bijection. Indeed, it is clear that for each open neighbourhood $\{f_n\colon n\geq m \}\cup \{{\boldsymbol{0}} \}$ of $\boldsymbol{0}=\phi(0)$,  $\phi(\{\frac{1}{k_n}\colon n\geq m \}\cup \{{0} \})=\{f_n\colon n\geq m \}\cup \{{\boldsymbol{0}} \}$ and $\{\frac{1}{k_n}\colon n\geq m \}\cup \{{0} \}$ is a neighbourhood of $0$. Thus, K is a compact set in $\mathcal{M}_\mu$. Also, it is clear that  $K$ is an infinite set.
	\end{proof}
	
	Thus, we can unite the above discussions as follows.
	
	\begin{theorem} \label{ccompact2}
		$\mu$ is bounded away from zero if and only if each compact set in $\mathcal{M}_\mu$ is at most finite.
	\end{theorem}
	
	We note that the compact set constructed while  proving Theorem \ref{tcompact2} is countably infinite. The question of  existence of Lindel\"{o}f sets in $\mathcal{M}_\mu$ which meets uncountably many $\{f\}$'s remains open. However, we partially answer this in the following result.
	
	\begin{theorem} \label{tLnon}
		Let $\mu$ be a measure which is not purely atomic. Then there exists an uncountable compact (and hence Lindel\"{o}f) set in $\mathcal{M}_\mu$.
	\end{theorem}

	\begin{proof}
		By Theorem \ref{tdecompose}, $\mu$ can be expressed as the sum of a purely atomic measure $\mu_1$ and a non-atomic measure $\mu_2$ such that $\mu_1\mathcal{S}\mu_2$ and $\mu_2\mathcal{S}\mu_1$. Since $\mu$ is not purely atomic, $\mu_2$ is a non-zero measure. Let $\mu_2(X)=b>0$. So, it follows from Theorem \ref{tnonatomic} that we can associate with each $r\in[0,b]$, an $A_r\in \mathscr{A}$ such that $\mu(A_r)=r$. Since $\mu_2\mathcal{S}\mu_1$, for each $r\in [0,b]$, there exists $F_r\in \mathscr{A}$ with $F_r\subseteq A_r$ such that $\mu_2(F_r)=\mu_2(A_r)$ and $\mu_1(F_r)=0$. Then $\mu(F_r)=r$ and using Lemma \ref{nonatomic}, we can assume without loss of generality that $F_0=\emptyset$, $F_b=X$ and whenever $r,s\in [0,b]$ with $r<s$, $F_r\subseteq F_s$. With each $r\in [0,b]$, we assign a measurable function $f_r\colon X\longrightarrow \mathbb{R}$ defined by $f_r(x)=\begin{cases}
			1 &if\;x\in F_r \\
			0 &otherwise
		\end{cases}$ and define $\phi\colon [0,b]\longrightarrow \mathcal{M}_\mu$ as $\phi(r)=f_r$. We assert that $\phi$ is a continuous injection. Indeed, whenever $r,s\in [0,b]$ with $s\neq r$, $\delta(f_s,f_r)=\begin{cases}
			\mu(F_s\setminus F_r) &if\;r<s\\
			\mu(F_r\setminus F_s) &if\;r>s
		\end{cases}=|s-r|$ and so $\phi$ is continuous. Note that $f_r\neq f_s$ whenever $s\neq r$. Thus, $\phi([0,b])$ is a compact set in $\mathcal{M}_\mu$ which is uncountable.
	\end{proof}

	Whether the assumption ``$\mu$ is not purely atomic" in the statement of Theorem \ref{tLnon}, can be substituted with the hypothesis that ``$\mu$ is bounded away from zero" remains an unanswered question and we raise it for the readers.
	
	\begin{question}
		Let $(X,\mathscr{A},\mu)$ be a purely atomic measure space where $\mu$ is not bounded away from zero, does there exist a Lindel\"{o}f set in $\mathcal{M}_\mu$ which is uncountable?
	\end{question}

	\section{Connectedness in $\mathcal{M}_\mu$}
	
	We aim to find out the connected components in the space $ \mathcal{M}_\mu$. Since $\mathcal{M}_\mu$ is a topological ring, the component of each point $f\in \mathcal{M}_\mu$ can be obtained by translating the component of $\boldsymbol{0}$. Thus, we only attempt to compute the component of $\boldsymbol{0}$. We first observe that $\mathcal{M}_\mu$ is a path-connected space whenever the underlying measure is non-atomic.

	\begin{theorem} \label{tconnected}
		Suppose $(X,\mathscr{A},\mu)$ is a non-atomic measure space. Then $\mathcal{M}_\mu$ is path connected $($and hence, connected$)$.
	\end{theorem}

	\begin{proof}
		Since $\mu$ is non-atomic, for each $r\in[0,1]$, there exists $A_r\in\mathscr{A}$ such that $\mu(A_r)=r$ (Theorem \ref{tnonatomic}).  By Lemma~\ref{nonatomic}, without loss of generality we can assume that $A_0=\emptyset$, $A_1=X$ and whenever $r,s\in (0,1)$ with $r<s$, $A_r\subseteq A_s$ with $\mu(A_r)=r$, $\mu(A_s)=s$.
		
		Consider $f,g\in \mathcal{M}_\mu$ with $f\neq g$ and define $\phi\colon [0,1]\longrightarrow \mathcal{M}_\mu$ as follows: \[\phi(r)(x)=\begin{cases}
			g(x) &if\;x\in A_r \\
			f(x) &if\;x\in X\setminus A_r
		\end{cases}. \] Then $\phi(0)=f$ and $\phi(1)=g$. Moreover, for each $\epsilon>0$ and $r\in [0,1]$, $\phi((r-\epsilon,r+\epsilon)\cap [0,1])\subseteq B(\phi(r),\epsilon)$. This ensures that $\phi$ is continuous on $[0,1]$. So, $f$ and $g$ are connected by a path. Thus, $\mathcal{M}_\mu$ is path connected.
	\end{proof}
	
	Thus, if $\mu$ non-atomic measure, then $\mathcal{M}_\mu$ has only one component. Naturally, we are curious about the components of $\mathcal{M}_\mu$ if $\mu$ is purely atomic. We attend to this in the next result.
	
	\begin{theorem} \label{tatomic}
		If $(X,\mathscr{A},\mu)$ is a purely atomic measure space, then $\mathcal{M}_\mu$ is totally disconnected.
	\end{theorem}
	
	\begin{proof}
		Let $f,g\in \mathcal{M}_\mu$ be such that $g\neq f$. Since $\mu$ is purely atomic, it follows from Theorem \ref{atomic} that $[0,1]\setminus \mu(\mathscr{A})$ is dense in $[0,1]$. So, there exists $\epsilon\in (0,\delta(f,g))\setminus \mu(\mathscr{A})$. Note that $B(f,\epsilon)=\{h\in \mathcal{M}\colon \delta(f,h)\leq \epsilon \}$ is a clopen set in $\mathcal{M}_\mu$ which contains $f$ but misses $g$. Therefore, $\{f\}$ is the component of $f$ in $\mathcal{M}_\mu$
	\end{proof}

	For an element $y$ in a topological space $Y$, the path component of $y$ is defined as the largest path connected set containing $y$. Since, path connected sets are always connected, it is evident that for a purely atomic measure $\mu$, the path components are the singleton sets. 
	Moreover, for a non-atomic measure space, since $\mathcal{M}_\mu$ is path connected, it is the only path component as well (Theorem \ref{tconnected}). The following corollary is therefore immediate.
	
	\begin{corollary}
		For a purely atomic or a non-atomic measure space, the components and path components in the space $\mathcal{M}_\mu$ agree.
	\end{corollary}

	So far, we have discussed only the cases involving either a non-atomic measure or a purely atomic measure. In order to discuss the situation where $\mu$ does not fall under either of the aforementioned cases, we first consider an example.

	\begin{example} \label{eg5.0}
		Consider the measure $\mu=\frac{1}{3}(\mu_l+2\delta_0)$ on $([0,1],\mathscr{L})$. Then $\mu$ is neither purely atomic nor non-atomic. 
		We first argue that the set $K_{\boldsymbol{0}}=\{f\in \mathcal{M}\colon \mu(X\setminus Z(f))\leq \frac{1}{3} \}$ is path-connected (and hence, connected). Indeed, if  $f\in K_{\boldsymbol{0}}$, then $\phi\colon [0,\frac{1}{3}]\longrightarrow K_{\boldsymbol{0}}$, defined as $\phi(r)(x)=\begin{cases}
			0 &if\;x\in A_r \\
			f(x) &if\;x\in X\setminus A_r
		\end{cases}$ constitutes a path in $K_{\boldsymbol{0}}$ which joins $\boldsymbol{0}$ and $f$; where $A_r=(0,r)$ for all $r\in (0,\frac{1}{3})$, $A_0=\emptyset$ and $A_{\frac{1}{3}}=X$. Now, note that $N(\mathcal{M}_\mu)=[0,\frac{1}{3}]\sqcup [\frac{2}{3},1]$ and $N$ is continuous. Therefore, $K_{\boldsymbol{0}}$ is the component of $\boldsymbol{0}$ in $\mathcal{M}_\mu$. Note that $\chi_{\{0\}}\in K_{\boldsymbol{0}}\setminus \{\boldsymbol{0}\}$ and $\boldsymbol{1}\in \mathcal{M}_\mu\setminus K_{\boldsymbol{0}}$.
	\end{example}
	
	Note that the non-atomic part and purely atomic part in the above example are $\mu_2=\frac{1}{3}\mu_l$ and $\mu_1=\frac{2}{3}\delta_0$ respectively; and the set $K_{\boldsymbol{0}}$ can also be expressed as $\{f\in \mathcal{M}\colon \mu_1(X\setminus Z(f))=0 \}$. Moreover, consider $\mu=\mu_1+\mu_2$ as in Theorem \ref{Johnson} and $K_{\boldsymbol{0}}=\{f\in \mathcal{M}\colon \mu_1(X\setminus Z(f))=0 \}$. Then, for a purely atomic measure, $K_{\boldsymbol{0}}=\{\boldsymbol{0}\}$ and for a non-atomic measure, $K_{\boldsymbol{0}}=\mathcal{M}_\mu$, which are the components of $\boldsymbol{0}$ in the respective cases.  It is therefore pertinent to ask if the component of $\boldsymbol{0}$ in $\mathcal{M}_\mu$ is always of the form $K_{\boldsymbol{0}}$. We answer this in the affirmative through the following result.
	
	\begin{theorem} \label{tcomponent}
		Let $\mu=\mu_1+\mu_2$ be a measure on a measurable space $(X,\mathscr{A})$, where $\mu_1$ is a purely atomic measure, $\mu_2$ a non-atomic measure, $\mu_1\mathcal{S}\mu_2$ and $\mu_2\mathcal{S}\mu_1$. Also let $K_{\boldsymbol{0}}=\{f\in \mathcal{M}\colon \mu_1(X\setminus Z(f))=0 \}$. Then the following assertions hold.
		\begin{enumerate}[label=\arabic*.]
			\item $K_{\boldsymbol{0}}$ is a path connected set in $\mathcal{M}_\mu$.
			\item For any $\epsilon\in [0,1]\setminus \mu_1(\mathscr{A})$, $B_1(\epsilon)=\{f\in \mathcal{M}\colon \mu_1(X\setminus Z(f))<\epsilon \}$ is a clopen set in $\mathcal{M}_\mu$.
			\item $K_{\boldsymbol{0}}$ is the component of  $\boldsymbol{0}$ in $\mathcal{M}_\mu$.
		\end{enumerate}
	\end{theorem}
	
	\begin{proof}
		\	\begin{enumerate}[label=\arabic*.]
			\item Let $f\in K_{\boldsymbol{0}}\setminus \{\boldsymbol{0} \}$. Define $b=\mu(X\setminus Z(f))$. Since $f\in K_{\boldsymbol{0}}$, $b=\mu(X\setminus Z(f))=\mu_2(X\setminus Z(f))$.
			In light of Theorem \ref{tnonatomic} and Lemma \ref{nonatomic}, for each $r\in (0,b)$, there exists $F_r\in \mathscr{A}$ with $F_r \subseteq X\setminus Z(f)$ such that $\mu_2(F_r)=r$ and whenever $r,s\in [0,b]$ with $r<s$, $F_r\subseteq F_s$. Moreover, take $F_0=\emptyset$ and $F_{b}=X\setminus Z(f)$. Define $\phi\colon [0,b]\longrightarrow K_{\boldsymbol{0}}$ such that $\phi(r)(x)=\begin{cases}
					0 &if\;x\in F_r \sqcup Z(f) \\
					f(x) &otherwise
				\end{cases}$, for each $r\in [0,b]$. Then $\phi(0)=f$ and $\phi(b)=\boldsymbol{0}$. For all $r,s\in [0,b]$, $\delta(r,s)\leq |r-s|$. This ensures that $\phi$ is continuous on $[0,b]$ and thus $\phi$ defines a path in $K_{\boldsymbol{0}}$ joining $\boldsymbol{0}$ and $f$. Therefore,  $K_{\boldsymbol{0}}$ is path connected.
		
			\item Let $g\in B_1(\epsilon)$ and choose a positive real number $\epsilon_1<\epsilon-\mu_1(X\setminus Z(g))$. Since $\mu_1(A)\leq \mu(A)$ for any $A\in \mathscr{A}$; it follows that $B(g,\epsilon_1)\subseteq B_1(\epsilon)$ and so $B_1(\epsilon)$ is open. Again let $h\notin B_1(\epsilon)$. Then as $\epsilon\notin \mu_1(\mathscr{A})$, $\mu_1(X\setminus Z(h))>\epsilon$. Now, choose a positive real number $\epsilon_2<\mu_1(X\setminus Z(h))-\epsilon$. It can be easily observed that $B(h,\epsilon_2)\cap B_1(\epsilon)=\emptyset$ and so $B_1(\epsilon)$ is closed as well.
			
			\item It is sufficient to show that for any $f\notin K_{\boldsymbol{0}}$, then there exists a clopen set in $\mathcal{M}_\mu$ which contains $K_{\boldsymbol{0}}$ and misses $f$. Indeed since $\mu_1(X\setminus Z(f))>0$ and $[0,1]\setminus \mu_1(\mathscr{A})$ is dense in $[0,1]$, there exists $\epsilon\in (0,\mu_1(X\setminus Z(f)))\setminus \mu_1(\mathscr{A})$. It is now clear that the clopen set $B_1(\epsilon)$ contains $K_{\boldsymbol{0}}$ but misses $f$.
		\end{enumerate}
	\end{proof}
	The following corollary follows immediately, as $\mathcal{M}_\mu$ is a topological ring.
	\begin{corollary}
		Considering the hypothesis of Theorem \ref{tcomponent}, for each $f\in \mathcal{M}$, the set $K_f=f+K_{\boldsymbol{0}}=\{g\in \mathcal{M}\colon \mu_1(X\setminus Z(f-g))=0 \}$ is the component of $f$ in $\mathcal{M}_\mu$.

	\end{corollary}

	Furthermore, as $K_{\boldsymbol{0}}$ is itself path connected and is the component of $\boldsymbol{0}$ in $\mathcal{M}$, the following conclusion is immediate.
	
	\begin{corollary} \label{cpath}
		$K_{\boldsymbol{0}}$ (resp. $K_f$) is the path component of $\boldsymbol{0}$ (resp. $f$) in $\mathcal{M}_\mu$.
	\end{corollary}

	We now revisit the definition of quasicomponent of a point. In a topological space $Y$, the quasicomponent of a point $y\in Y$ is defined to be the intersection of all clopen sets in $Y$, containing $y$. In general, the quasicomponent of a point $y$ contains the component of $y$ in $Y$, which in turn contains the path component of $y$. We realise in the next result that these three notions coincide in $\mathcal{M}_\mu$.
	
	\begin{theorem} \label{tquasi}
		For any measure space $(X,\mathscr{A},\mu)$, then the quasicomponent, component and path component of each point in $\mathcal{M}_\mu$ coincide. 
	\end{theorem}
	
	\begin{proof}
		The fact that the path component and component of each point in $\mathcal{M}_\mu$ coincide follows from the fact that the components in this space are itself path connected. Furthermore, the component of $\boldsymbol{0}$ in $\mathcal{M}_\mu$ can be expressed as: $K_{\boldsymbol{0}}=\bigcap\limits_{\epsilon\in [0,1]\setminus \mu_1(\mathscr{A})}B_1(\epsilon)$, where each $B_1(\epsilon)$ is clopen in $\mathcal{M}_\mu$. Therefore, $K_{\boldsymbol{0}}$ (and resp. $K_f$) is also the quasicomponent of $\boldsymbol{0}$ (resp. $f$) in $\mathcal{M}_\mu$.
	\end{proof}
	
	We must note that if $\mathcal{M}_\mu$ is connected, then $K_{\boldsymbol{0}}=\mathcal{M}_\mu$.  Therefore, $\mu_1(X\setminus Z(\boldsymbol{1}))=0$. That is, $\mu_1(X)=0$ and so $\mu$ is non-atomic. Combining this observation with Theorem \ref{tconnected}, we present the following theorem, which characterises the non-atomicity of $\mu$.
	
	\begin{theorem} \label{tnon}
		For a measure space $(X,\mathscr{A},\mu)$, the space $\mathcal{M}_\mu$ is connected if and only if $\mu$ is a non-atomic measure.
	\end{theorem}

	
	Moreover, purely atomic measures can also be characterised in a similar manner, as has been noted in the following result.
	
	\begin{theorem} \label{tatomiccomp}
		Assume the hypothesis of Theorem \ref{tcomponent}. Then, the component of $\boldsymbol{0}$, $K_{\boldsymbol{0}}= \{\boldsymbol{0}\}$ (or, any $K_f=\{f\}$) if and only if $\mu$ is purely atomic. In other words, $\mathcal{M}_\mu$ is totally disconnected if and only if  $\mu$ is purely atomic.
	\end{theorem}
	
	\begin{proof}
		Assume that $\mu$ is not purely atomic. Then the non-atomic part, $\mu_2$ is a non-trivial measure. Since $\mu_2\mathcal{S}\mu_1$, there exists $F\in \mathscr{A}$ with $\mu_2(F)=\mu_2(X)>0$ and $\mu_1(F)=0$. Consider $f=\chi_F$. Then $\mu_1(X\setminus Z(f))=\mu_1(F)=0$ and $\mu(X\setminus Z(f))=\mu(F)=\mu_2(F)>0$. Therefore, $f\in K_{\boldsymbol{0}}\setminus \{\boldsymbol{0}\}$. The converse follows from Theorem \ref{tatomic}.
	\end{proof}

	A topological space $Y$ is said to be zero-dimensional if it has a clopen base. 
	In this context, we deduce a necessary and sufficient condition for the space $\mathcal{M}_\mu$ to be zero-dimensional, via the atomicity of the underlying measure $\mu$.
	
	\begin{theorem} \label{r0}
		$\mathcal{M}_\mu$ is a zero-dimensional space if and only if $\mu$ is purely atomic. 
	\end{theorem}
	\begin{proof}
		Let us suppose that $\mu$ is purely atomic. By Lemma \ref{atomic}, $[0,1]\setminus \mu(\mathscr{A})$ is dense in $[0,1]$. Then the  collection $\{B(f,\epsilon)\colon f\in \mathcal{M}_\mu,\; \epsilon>0\;with\;\epsilon\notin \mu(\mathscr{A}) \}$ forms a clopen base for the space $\mathcal{M}_\mu$. 
		
		Suppose $\mathcal{M}_\mu$ is zero-dimensional and if possible let $\mu$ be not purely atomic. Then it follows from Theorem \ref{tatomiccomp} that $\{\boldsymbol{0}\}\subsetneqq K_{\boldsymbol{0}}$. Choose $f\in K_{\boldsymbol{0}}\setminus \{\boldsymbol{0}\}$. Since $\{\boldsymbol{0}\}$ is closed, there exists a clopen set $K$ in $\mathcal{M}_\mu$ such that $\{\boldsymbol{0}\}\subseteq K$ and $f\notin K$. Therefore, $K\cap K_{\boldsymbol{0}}$ is a non-trivial clopen set in $K_{\boldsymbol{0}}$, which contradicts that $K_{\boldsymbol{0}}$ is connected.
	\end{proof}

	The above results can be consolidated as follows:
	
	\begin{theorem} \label{t3.13}
		The following statements are equivalent for a measure space $(X,\mathscr{A},\mu)$.
		
		
		\begin{enumerate}[label=\arabic*.]
			\item $\mu$ is purely atomic.
			\item $\mathcal{M}_\mu$ is totally disconnected.
			\item $\mathcal{M}_\mu$ is zero-dimensional.
			
		\end{enumerate}
		
	\end{theorem}

	We shall terminate this section with the discussion of the notion of local connectedness of the space $\mathcal{M}_\mu$. That $\mathcal{M}_\mu$ is locally connected, whenever $\mu$ is bounded away from zero is obvious, as the space then becomes discrete. However, for a purely atomic measure $\mu$ which is not bounded away from zero, it follows from Theorem \ref{tatomic} that $\mathcal{M}_\mu$ is not locally connected. Moreover, for a non-atomic measure space, $\mathcal{M}_\mu$ is connected, and hence locally connected. We consolidate these ideas in a more general setting in the next result.
	
	\begin{theorem} \label{tlocal}
		Consider $\mu=\mu_1+\mu_2$, where $\mu_1$ and $\mu_2$ are as described in Theorem \ref{tdecompose}. Then $\mathcal{M}_\mu$ is locally connected if and only if the purely atomic part, $\mu_1$ is either zero or is bounded away from zero. 
	\end{theorem}
	
	\begin{proof}
		
		If $\mu_1$ is the zero measure, then $\mu$ is non-atomic and thus $\mathcal{M}_\mu$ is connected. 
		
		Assume that $\mu_1$ is bounded away from zero and let $\lambda>0$ be such that for every positive measurable set $A$, $\mu_1(A)> \lambda$. We assert that $K_{\boldsymbol{0}
		}$ (resp. each $K_f$) is open in $\mathcal{M}_\mu$. Let $f\in K_{\boldsymbol{0}}$ and $h\in B(f,\lambda)$. Then $\mu(X\setminus Z(f-h))<\lambda$ and so $\mu_1(X\setminus Z(f-h))<\lambda$. This ensures that $\mu_1(X\setminus Z(f-h))=0$. Since $f\in K_{\boldsymbol{0}}$, $\mu_1(X\setminus Z(f))=0$ and as a result $\mu_1(X\setminus Z(h))=0$. Hence, $h\in K_{\boldsymbol{0}}$.
		
		Conversely suppose $\mu_1$ is not bounded away from zero. If possible let, $\mathcal{M}_\mu$ be locally connected and $K$ be a connected neighbourhood of $\boldsymbol{0}$ in $\mathcal{M}_\mu$. Then there exists $\epsilon>0$ such that $B(\boldsymbol{0},\epsilon)\subseteq K\subseteq K_{\boldsymbol{0}}$. Since $\mu_1$ is not bounded away from zero, there exists $A\in \mathscr{A}$ such that $0<\mu_1(A)<\epsilon$. As $\mu_1\mathcal{S}\mu_2$, there exists $F\in \mathscr{A}$ such that $\mu_1(A)=\mu_1(F)$ and $\mu_2(F)=0$. Therefore, $\mu(F)=\mu_1(F)<\epsilon$ and so $\chi_F\in B(\boldsymbol{0},\epsilon)$ but as $\mu_1(F)>0$, $\chi_F\notin K_{\boldsymbol{0}}$, which contradicts that $B(\boldsymbol{0},\epsilon)\subseteq K_{\boldsymbol{0}}$.
	\end{proof}
	
	\begin{remark}
		We recall  at this point that components and quasicomponents of a topological space $Y$ coincide if $Y$ is locally connected or is compact and Hausdorff. However, for a measure whose purely atomic part is not bounded away from zero (see Example \ref{eg}(\ref{eg0})), $\mathcal{M}_\mu$ provides an example of a space which is neither locally connected (Theorem \ref{tlocal}) nor compact (Corollary \ref{ccompact}) and yet the notions of components and quaiscomponents coincide (Theorem \ref{tquasi}).
	\end{remark}

\end{document}